\documentclass[11pt,reqno,oneside]{amsart}
\usepackage[utf8]{inputenc}
\usepackage{amsmath,amsthm,amssymb}
\usepackage{colonequals} 
\usepackage{graphics}
\usepackage[pagebackref=true]{hyperref}
\usepackage[usenames, dvipsnames]{xcolor}
\usepackage[normalem]{ulem}
\usepackage{booktabs}
\usepackage{mathtools}
\mathtoolsset{showonlyrefs}
\usepackage{enumitem}
\newlist{todolist}{itemize}{2}
\setlist[todolist]{label=$\square$}
\usepackage{pifont}

\usepackage[square,sort,comma,numbers]{natbib}
\definecolor{darkblue}{rgb}{0.0,0.0,0.3}
\hypersetup{colorlinks,breaklinks,
  linkcolor=darkblue,urlcolor=darkblue,
anchorcolor=darkblue,citecolor=darkblue}
\theoremstyle{plain}
\newtheorem{theorem}{Theorem}
\newtheorem*{theorem*}{Theorem}
\newtheorem{lemma}[theorem]{Lemma}
\newtheorem{proposition}[theorem]{Proposition}
\newtheorem*{proposition*}{Proposition}

\newtheorem*{corollary*}{Corollary}

\theoremstyle{definition}
\newtheorem{remark}[theorem]{Remark}

\newtheorem*{algorithm*}{Algorithm}

\newtheorem*{experiment*}{Experiment}

\DeclareMathOperator{\Cl}{Cl}
\DeclareMathOperator{\CT}{CT}

\renewcommand{\Im}{\operatorname{Im}}
\renewcommand{\Re}{\operatorname{Re}}
\DeclareMathOperator{\im}{Im}

\DeclareMathOperator{\Gal}{Gal}
\DeclareMathOperator{\GL}{GL}
\DeclareMathOperator{\Nm}{Nm}

\DeclareMathOperator*{\Res}{Res}

\DeclareMathOperator{\Tr}{Tr}

\DeclareMathOperator{\new}{new}
\DeclareMathOperator{\cond}{cond}

\newcommand{\C}{\mathbb{C}}
\newcommand{\F}{\mathbb{F}}
\newcommand{\Q}{\mathbb{Q}}
\newcommand{\R}{\mathbb{R}}
\newcommand{\Z}{\mathbb{Z}}

\newcommand{\sbar}{\overline{s}}

\newcommand{\calE}{\mathcal{E}}
\newcommand{\calH}{\mathcal{H}}
\newcommand{\calO}{\mathcal{O}}

\newcommand{\frakb}{\mathfrak{b}}
\newcommand{\frakp}{\mathfrak{p}}

\newcommand{\Zeven}{Z_D^{\textup{even}}}
\newcommand{\Zodd}{Z_D^{\textup{odd}}}

\newcommand{\chid}{\chi_{4D}}
\newcommand{\Szero}[2]{L^2(\Gamma_0(#1)\backslash\mathcal{H}; #2)}
\newcommand{\Sone}[1]{L^2(\Gamma_1(#1)\backslash\mathcal{H})}

\title{The Fibonacci Zeta Function and Modular Forms}

\author[Assaf]{Eran Assaf}
\author[Kuan]{Chan Ieong Kuan}

\author[Lowry-Duda]{David Lowry-Duda}

\author[Walker]{Alexander Walker}

\date{Last compiled: \today}

\begin{document}

\begin{abstract}
  We show that a family of Dirichlet series generalizing the Fibonacci zeta function $\sum F(n)^{-s}$ has meromorphic
  continuation in terms of dihedral $\mathrm{GL}(2)$ Maass forms.
\end{abstract}

\maketitle

\section{Introduction}%
\label{sec:intro}

In~\cite{akldwFibonacciGeneral}, the authors introduced a family of zeta
functions that generalized earlier work~\cite{egami1999curious,
  navas2001fibonacci, kamano2013analytic}
on the Fibonacci zeta function
\begin{equation}
  Z_{\mathrm{Fib}}(s)
  :=
  \sum_{n \geq 1} \frac{1}{F(n)^s}
  =
  1 + 1 + \frac{1}{2^s} + \frac{1}{3^s} + \ldots
\end{equation}

For a square-free integer $D > 1$, let $\calO_D$ denote the
ring of integers in $\mathbb{Q}(\sqrt{D}) \subset \mathbb{R}$.
Let $\varepsilon > 1$ denote the fundamental unit in $\calO_D$.
To simplify presentation and the shape of results, we suppose that
$\varepsilon$ satisfies $N(\varepsilon) = -1$ (which occurs if and only if the
class group and narrow class group of $\mathbb{Q}(\sqrt{D})$ coincide).

We define the $\calO_D$ Lucas and Fibonacci sequences in terms of
traces of the fundamental unit,
\begin{align}
  L_{D}(n) & := \Tr_{\calO_D}(\varepsilon^n),            %
  \label{align:lucas_def}%
  \\
  F_{D}(n) & := \Tr_{\calO_D}(\varepsilon^n / \sqrt{q}), %
  \label{align:fib_def}%
\end{align}
where $q = D$ if $D \equiv 1 \bmod 4$ and otherwise $q = 4D$.
Note that $F_5(n) = F(n)$, the standard Fibonacci function.

The main results of~\cite{akldwFibonacciGeneral} give meromorphic continuation
to the $\mathcal{O}_D$ Fibonacci zeta functions
\begin{align}
  \Zodd(s)
   & :=
  \sum_{n \geq 1}
  \frac{1}{F_D(2n - 1)^s}
  =
  \sum_{n \geq 1}
  \frac{1}{\big(\Tr_{\calO_D}(\varepsilon^{2n - 1} / \sqrt{q}) \big)^s},
  \\
  \Zeven(s)
   & :=
  \sum_{n \geq 1}
  \frac{1}{F_D(2n)^s}
  =
  \sum_{n \geq 1}
  \frac{1}{\big(\Tr_{\calO_D}(\varepsilon^{2n} / \sqrt{q}) \big)^s}
\end{align}
using binomial expansion or Poisson summation.
A third proof was sketched that related the $\mathcal{O}_D$ Fibonacci zeta
function to modular forms.

In this article, we complete this sketch and describe how to understand the
$\mathcal{O}_D$ Fibonacci zeta functions via modular forms.
Ultimately, we show that the resulting meromorphic continuation agrees with the
continuation obtained via Poisson summation in Theorems~11 and~13 of~\cite{akldwFibonacciGeneral}.

\begin{theorem}\label{thm:main-theorem}
  The $\mathcal{O}_D$ Fibonacci zeta functions admit meromorphic continuation to $s \in \mathbb{C}$. We have
  \begin{align*}
    \Zodd(s) & =
    \frac{q^{s/2}}{8 \Gamma(s) \log \varepsilon}
    \sum_{m \in \mathbb{Z}} (-1)^m
    \Gamma\Big(\frac{s}{2} + \frac{\pi i m}{2\log \varepsilon}\Big)
    \Gamma\Big(\frac{s}{2} - \frac{\pi i m}{2\log \varepsilon}\Big), \\
    \Zeven(s)
             & =
    \frac{q^{\frac{s}{2}} \Gamma(1-s)}{4 \log \varepsilon} \sum_{m \in \mathbb{Z}}
    \frac{\Gamma(\frac{s}{2} - \frac{\pi i m}{2 \log \varepsilon})}
    {\Gamma(1-\frac{s}{2} - \frac{\pi i m}{2 \log \varepsilon})},
  \end{align*}
  for all $s \in \mathbb{C}$ and for $\Re s < 0$, respectively.
\end{theorem}

The expressions for $\Zodd$ and $\Zeven$ in Theorem~\ref{thm:main-theorem} arise as spectral expansions of a certain Petersson inner product.
Interestingly, these spectral expansions turn out to be supported only on dihedral Maass cusp forms.
From a representation-theoretic perspective, this implies that the associated
automorphic representation is dihedral.
It would be nice to have an argument that the representation $\Pi$ of $\GL_2$
assocaited to $V_1(z)$ is dihedral from first-principles.
But this has eluded the authors.

Though $\Zodd$ and $\Zeven$ are similar in many ways, $\Zodd$ is much simpler
to study than $\Zeven$. To highlight the main ideas of our method, we focus
mostly on $\Zodd$ in this article. A brief treatment of $\Zeven$ is presented
in \S\ref{rem:Z-even}.

\section*{Acknowledgements}

EA was supported by Simons Collaboration Grant (550029, to Voight) and
by Simons Foundation Grant (SFI-MPS-Infrastructure-00008651, to Sutherland).
CIK was supported in part by NSFC (No.\ 11901585).
DLD was supported by the Simons Collaboration in Arithmetic Geometry, Number
Theory, and Computation via the Simons Foundation grant 546235.
AW was supported by the Additional Funding Programme for Mathematical Sciences,
delivered by EPSRC (EP/V521917/1) and the Heilbronn Institute for Mathematical Research.

We would also like to thank John Voight, Andy Booker, Jeff Hoffstein, and Min Lee
for helpful discussion and feedback on the ideas leading into this paper.

\section{Shifted Convolution Dirichlet Series}

The connection between the Fibonacci zeta function and modular forms comes from
the following proposition, which recognizes $\mathcal{O}_D$ Fibonacci numbers
as the solutions to certain generalized Pell equations.

\begin{proposition}[Proposition~6 of~\cite{akldwFibonacciGeneral}]\label{prop:pfib}
  The positive integer $n$ is an $\calO_D$ Fibonacci number if and only if
  there is an integer solution in $X$ to the equation
  \begin{equation}\label{eq:proppfib}
    X^2 = qn^2 \pm 4,
  \end{equation}
  where $q = D$ if $D \equiv 1 \bmod 4$ and otherwise $q = 4D$.
  If in addition $N(\varepsilon) = -1$, then $n$ is an odd-indexed $\calO_D$
  Fibonacci number if and only if there is an integer solution in $X$ to the equation
  \begin{equation}
    X^2 = qn^2 - 4.
  \end{equation}
\end{proposition}

This proposition shows that an $\calO_D$ Fibonacci number can
be recognized by detecting squares.
Let $r_1(n) = \#\{x \in \Z :  x^2 = n\}$ (which is twice the square indicator
function, except $r_1(0) = 1$).
It follows that
\begin{align}
  \Zodd(s)  & = \sum_{n \ge 1} \frac{1}{F_D(2n-1)^s}
  = \frac{1}{4} \sum_{n \ge 1} \frac{r_1(n)r_1(qn-4)}{n^{s/2}},
  \\
  \Zeven(s) & = \sum_{n \ge 1} \frac{1}{F_D(2n)^s}
  = \frac{1}{4} \sum_{n \ge 1} \frac{r_1(n)r_1(qn+4)}{n^{s/2}}.
\end{align}
We simplify further depending on $D \bmod 4$.
When $D \equiv 1 \bmod 4$, we have $q = D$.
Otherwise, we have $q = 4D$ and $r_1(qn - 4) = r_1(4D - 4) = r_1(D - 1)$.
Stated differently, $4D - 4$ is a square if and only if $D - 1$ is a square.
Altogether, this allows us to rewrite the $\mathcal{O}_D$ Fibonacci series as
\begin{align}\label{eq:basic_modular}
  \Zodd(s) = \frac{1}{4} \sum_{n \ge 1} \frac{r_1(n)r_1(Dn-\ell)}{n^{s/2}},
  \; \Zeven(s) = \frac{1}{4} \sum_{n \ge 1} \frac{r_1(n)r_1(Dn+\ell)}{n^{s/2}},
\end{align}
where $\ell = \ell(D) = 4$ if $D \equiv 1 \bmod 4$ and otherwise $\ell = 1$.

The sequence $\{r_1(n)\}$ gives the coefficients of a modular form and the sums
on the right of~\eqref{eq:basic_modular} are shifted convolution Dirichlet
series.
We will examine $\Zodd$ in detail and discuss the necessary modifications to treat
$\Zeven$ in Remark~\ref{rem:Z-even}.

The series for $\Zodd$ in~\eqref{eq:basic_modular} is closely related to a
series studied in~\cite[\S3]{hkldw_3aps}.
As in that work, we obtain the meromorphic continuation for $\Zodd(s)$ via spectral
expansion.

Let $\theta(z) = \sum_{n \ge 0} r_1(n) e^{2 \pi i n z} \in S_{1/2}(\Gamma_0(4))$ denote
the classical theta function, and consider
$V_1(z) = \Im(z)^{1/2} \theta(Dz) \overline{\theta(z)}$.
Then $V_1(z)$ is an automorphic form of weight $0$ for $\Gamma_0(4D)$ with nebentypus
$\chid = \left( \frac{4D}{\cdot} \right)$.
For $h \ge 1$, let $P_h(z, s; \chid)$ denote the Poincar{\'e} series
\begin{equation} \label{eq:P-h-definition}
  P_h(z, s; \chid) \colonequals
  \frac{1}{2} \sum_{\gamma \in \Gamma_{\infty} \backslash \Gamma_0(4D)}
  \chid(\gamma) \Im(\gamma z)^s e^{2 \pi i h \gamma z},
\end{equation}
where $\Gamma_\infty \subset \Gamma_0(4D)$ denotes the stabilizer of the cusp at $\infty$.
Then $P_h$ is also an automorphic form of weight $0$ for $\Gamma_0(4D)$ with
nebentypus $\chid$.
The Petersson inner product of $V_1$ with $P_\ell$ (where $\ell =
  \ell(D)$ is $1$ or $4$, as described above) gives a completed version
of $\Zodd$:
\begin{equation}\label{eq:spectral_base}
  4 \Gamma(s) \Zodd(2s)
  = (4D \pi)^{s}
  \bigl\langle V_1, P_\ell(\cdot, \sbar + \tfrac{1}{2}; \chid) \bigr\rangle.
\end{equation}
The Poincar\'e series $P_\ell(\cdot, s)$ is known to have meromorphic
continuation to $\mathbb{C}$ (see for instance~\cite[\S15]{iwanieckowalski04}) and
decays rapidly at cusps.

\section{Spectral Expansion}

Our basic strategy is to obtain a spectral expansion of $\Zodd$
from the spectral expansion of $P_\ell$.
There is a minor obstruction: $V_1$ is non-cuspidal, which leads to poor convergence in the spectral expansion.
Thus we first modify $V_1$ by subtracting an Eisenstein series that cancels the
growth at the cusps in \S\ref{ssec:subt}.
We denote this by $V(z) = V_1(z) - E(z, \tfrac{1}{2})$.

Subsequently, in \S\ref{ssec:init_spectral} we spectrally expand $P_\ell$ and
insert this into~\eqref{eq:spectral_base}.
This expresses $\Zodd$ as a sum of terms indexed by the discrete and continuous
spectra of the hyperbolic Laplacian.
A short computation in~\S\ref{ssec:cont} shows that the contribution of the continuous spectrum vanishes.

Describing the discrete spectrum requires more work.
We show that the only forms that contribute are dihedral forms.
When $D \not \equiv 1 \bmod 4$, this is largely straightforward and follows
from a short argument given at the start of \S\ref{ssec:disc}.
But when $D \equiv 1 \bmod 4$, there is tension because there are oldforms that
appear in the spectral decomposition.
The problem is that the behavior seems easiest to understand by first studying
newforms on lower levels and then raising and orthogonalizing these forms in
full level.
We study oldform and newform Maass forms in \S\ref{ssec:disc}, and then
orthogonalize to get an orthogonal basis in \S\ref{ssec:orth}.

\subsection{Subtracting growth at cusps}\label{ssec:subt}

The process of mollifying growth at cusps by subtracting a linear combination
of Eisenstein series is one form of automorphic regularization.
Automorphic regularization is often described using a set of Eisenstein series
indexed by the cusps of $\Gamma_0(4D)$.
We avoid this approach here because the constant coefficients of the set
$\{\mathcal{E}_\mathfrak{a}(z,s)\}$ are linearly dependent functions of $z$
when $s = 1/2$.

Instead, we show that in this case it is still possible to mollify the growth of a certain real subspace of automorphic forms.

\begin{proposition}\label{prop:Fricke-regularizable}
  Let $D$ be square-free. Let $f$ be an automorphic form of weight $0$ for $\Gamma_0(4D)$ of nebentypus $\chid$, such that $W_{4D} f = \overline{f}$, where $W_{4D}$ is the Fricke involution.
  There exists an Eisenstein series $E(z,s)$, modular for $\Gamma_0(4D)$ of nebentypus $\chid$, such that
  \begin{equation}
    f(z) - E(z, \tfrac{1}{2})
    \in L^2(\Gamma_0(4D) \backslash \calH; \chid).
  \end{equation}
\end{proposition}

\begin{proof}
  Let $X$ be a set of representatives for the singular cusps of $\chi_{4D}$ under the action of $W_{4D}$. Then the constant term map $\CT : \calE( \cdot, \tfrac{1}{2}) \to \C^{X}$ is surjective, where $\calE$ is the space spanned by the Eisenstein series and $\CT(f)_x$ is the constant term of $f \vert_{\sigma_x}$. We further consider the decomposition to real vector spaces $\calE = \calE_+ \oplus \calE_-$, where $\calE_+$ consists of $E$ such that $W_{4D} E = \overline{E}$ and $\calE_-$ of those with $W_{4D} E = - \overline{E}$. These induce a decomposition $\C^X = V_+ \oplus V_-$.
  It follows that $\CT(f) \in V_+$, so there exists an Eisenstein series $E \in \calE_+$ such that $\CT(E(\cdot, \tfrac{1}{2})) = \CT(f)$.
  Finally, if $x \in X$, then the constant term of $f$ at $W_{4D}(x)$ is
  \begin{align*}
    \CT(W_{4D} f)_{x} & = \CT(\overline{f})_{x} = \overline{\CT(f)_{x}} = \overline{\CT(E(\cdot, \tfrac{1}{2}))_{x}} \\
                      & = \CT(\overline{E(\cdot, \tfrac{1}{2})})_{x} = \CT((W_{4D} E)(\cdot, \tfrac{1}{2}))_{x},
  \end{align*}
  which is the constant term of $E(\cdot, \tfrac{1}{2})$ at $W_{4D}(x)$, showing the result.
\end{proof}

\begin{remark}\label{rem:explicit-regularization}
  Since $W_{4D}(V_1) = \overline{V_1}$, Proposition~\ref{prop:Fricke-regularizable} implies that there exists an Eisenstein series $E(z,s)$, modular for $\Gamma_0(4D)$ of nebentypus $\chid$, for which
  \begin{equation}
    V(z)
    \colonequals
    V_1(z) - E(z, \tfrac{1}{2})
    \in L^2(\Gamma_0(4D) \backslash \calH; \chid).
  \end{equation}
  To explicitly construct such an Eisenstein series, we use a linear combination of Eisenstein series indexed by two characters, $E_{\chi_1, \chi_2}$, as defined in~\cite{huxley0},
  \begin{equation}
    E_{\chi_1, \chi_2}(z,s)
    =
    \frac{1}{2} \sum_{\gamma \in \Gamma_{\infty} \backslash \Gamma_0(4D)}
    \chi_1(\gamma_{2,1}) \chi_2(\gamma_{2,2}) \Im (\gamma z)^s,
    \quad \gamma = \begin{pmatrix} \gamma_{1,1} & \gamma_{1,2} \\ \gamma_{2,1} &
                \gamma_{2,2}\end{pmatrix}.
  \end{equation}
  For an odd divisor $r \mid D$, define $\chi_{r^\diamond}$ to be the unique
  primitive quadratic character of conductor $r$.
  Stated differently, write $r^\diamond= (\frac{-1}{r}) r$, where
  $\left(\frac{-1}{\cdot}\right)$ is the Kronecker symbol.
  Then $\chi_{r^\diamond}$ is the Dirichlet character that agrees with the
  Kronecker symbol $\left(\frac{r^\diamond}{\cdot}\right)$.
  We also use Shimura's symbol $\varepsilon_r$, which equals $1$ or $i$ according as $r \equiv 1$ or $3 \pmod 4$.
  If $D \equiv 2, 3 \bmod 4$, we define
  \begin{equation}
    E(z,s)
    =
    \sum_{\substack{0< r \mid D \\ 2 \nmid r}}
    \frac{\varepsilon_r}{\sqrt{r}} E_{\chi_{r^\diamond}, \chid
        \chi_{r^{\diamond}}}(z,s).
  \end{equation}
  Here, by abuse of notation, we suppose that $\chid \chi_{r^\diamond}$ denotes the primitive character inducing $\chid \chi_{r^\diamond}$.
  If $D \equiv 1 \bmod 4$, we define
  \begin{equation}
    \beta_{r,1} = \beta_{r,4}
    = \begin{cases}
      1,           & r \equiv \frac{D}{r} \bmod 8      \\
      \frac{1}{3}, & r \not \equiv \frac{D}{r} \bmod 8
    \end{cases},
    \qquad
    \beta_{r,2}
    = \begin{cases}
      -\sqrt{2} \left(\frac{2}{r} \right), & r \equiv \frac{D}{r} \bmod 8      \\
      0,                                   & r \not \equiv \frac{D}{r} \bmod 8
    \end{cases},
  \end{equation}
  and then define
  \begin{equation}\label{eq:eisenstein_lincomb}
    E(z,s)
    =
    \sum_{\substack{0 < r \mid D \\ r^2 \le D}}
    \frac{\varepsilon_r}{\sqrt{r}}
    \sum_{B \mid 4} \beta_{r,B} E_{\chi_{r^\diamond}, \chid \chi_{r^\diamond}}(Bz,s).
  \end{equation}
  In both cases, one can check using explicit formulas (e.g.\ those in~\cite{Young}) that $V_1(z) - E(z, \frac{1}{2})$ vanishes at every cusp of
  $\Gamma_0(4D)$ that is singular for the character $\chid$.
\end{remark}

\subsection{Initial Spectral Decomposition}\label{ssec:init_spectral}

We are now ready to spectrally expand $P_\ell$ for use in~\eqref{eq:spectral_base}.
Let $\{ \mu_j \}$ be an orthonormal basis of Hecke--Maass forms  
(cf.~\cite{Maass}) for $L^2(\Gamma_0(4D) \backslash \calH; \chid)$, with
respective spectral types $\frac{1}{2} + it_j$.
Each $\mu_j$ admits a Fourier expansion of the form
\begin{equation} \label{eq: Fourier expansion of cusp forms}
  \mu_j(x+iy)
  =
  \sqrt{y} \sum_{n \ne 0} \rho_j(n) K_{it_j} (2 \pi |n| y) e^{2 \pi i n x},
\end{equation}
with associated Hecke eigenvalues $\lambda_j(n)$.
By orthonormality, we can write $\rho_j(n) = \rho_j(1) \lambda_j(n)$ for all $n$.
We assume $\rho_j(1) \in \R$ without loss of generality.

We also let $\{ \mathcal{E}_\mathfrak{a} \}$ be a basis of Eisenstein series
associated to the cusps $\{ \mathfrak{a} \}$ of $\Gamma_0(4D)$ that are singular with respect to
$\chid$.
The Eisenstein series $\mathcal{E}_\mathfrak{a}$ is given by
\[
  \mathcal{E}_\mathfrak{a}(z,w; \chid)
  :=
  \sum_{\gamma \in \Gamma_\mathfrak{a} \backslash \Gamma_0(4D)}
  \chi_{D}(\gamma) \Im (\sigma_{\mathfrak{a}}^{-1} \gamma z)^w,
\]
where $\sigma_{\mathfrak{a}}$ is a scaling matrix associated to the cusp $\mathfrak{a}$.

Then by~\cite[Prop.\ 4.1, 4.2]{dfi}, $P_\ell$ has a spectral expansion of the form
\begin{align}
  \label{eq:P4-spectral expansion}
  P_\ell(z, s; \chid)
   & =
  \sum_j
  \bigl\langle P_\ell(\cdot, s; \chid), \mu_j \bigr\rangle \mu_j(z)
  \\
   & \quad +
  \sum_{\mathfrak{a}} \frac{1}{4 \pi i}
  \int_{(\tfrac{1}{2})}
  \bigl\langle
  P_\ell(\cdot, s; \chid),
  \mathcal{E}_\mathfrak{a}(\cdot, w; \chid)
  \bigr\rangle
  \mathcal{E}_\mathfrak{a}(z, w; \chid) dw,
\end{align}
where $\mathfrak{a}$ indexes the singular cusps.
From~\eqref{eq:spectral_base} and Remark~\ref{rem:explicit-regularization}, we have
\begin{equation} \label{eq:regularized-expansion}
  4 \Gamma(s) \Zodd(2s)
  =
  (4 \pi D)^s \bigl\langle V, P_\ell(\cdot, \overline{s}+\tfrac{1}{2}) \bigr\rangle
  +
  (4 \pi D)^s \bigl\langle E(\cdot, \tfrac{1}{2}), P_\ell(\cdot, \overline{s}+\tfrac{1}{2}) \bigr\rangle.
\end{equation}
For $\Re s$ sufficiently large, inserting the spectral expansion for $P_\ell$
produces the following spectral expansion for $\Zodd(2s)$:
\begin{align}
  \label{eq:initial_spectral}
  \frac{4 \Gamma(s) \Zodd(2s)}{(4\pi D)^s}
   &
  =
  \bigl\langle E(\cdot, \tfrac{1}{2}), P_\ell(\cdot, \overline{s}+\tfrac{1}{2}) \bigr\rangle
  +
  \sum_j \langle P_\ell(\cdot, \overline{s}+\tfrac{1}{2}), \mu_j \rangle
  \langle V, \mu_j \rangle
  \\
   & \quad
  + \sum_{\mathfrak{a}} \frac{1}{4 \pi i}
  \int_{(\tfrac{1}{2})}
  \bigl\langle
  P_\ell(\cdot, \overline{s}+\tfrac{1}{2}),
  \mathcal{E}_\mathfrak{a}(\cdot, w)
  \bigr\rangle
  \bigl\langle
  V,
  \mathcal{E}_\mathfrak{a}(\cdot, w; \chid)
  \bigr\rangle
  dw.\notag{}
\end{align}

\subsection{Vanishing in the Continuous Spectrum}\label{ssec:cont}

Surprisingly, most of the terms of the spectral expansion
in~\eqref{eq:initial_spectral} vanish.
We first show that the second line in~\eqref{eq:initial_spectral}, coming from
Eisenstein series $\mathcal{E}_\mathfrak{a}$ in the continuous spectrum,
vanishes entirely.

\begin{lemma}\label{lemma:continuous_vanish}
  For each singular cusp $\mathfrak{a}$,
  $\langle V(z), \mathcal{E}_\mathfrak{a}(z,\overline{w}; \chid) \rangle = 0$.
\end{lemma}

\begin{proof}
  We apply the Rankin--Selberg method to  
  unfold the inner product into the integral
  \[
    \langle
    V,
    \mathcal{E}_\mathfrak{a}(z,\overline{w}; \chid)
    \rangle
    =
    \int_0^\infty \int_0^1
    V(\sigma_\mathfrak{a}z) y^w dx \frac{dy}{y^2}.
  \]
  The integral in $x$ isolates the constant term of $V(\sigma_{\mathfrak{a}}z)$.
  We claim that this constant term is actually $0$, which implies the inner
  product is zero.

  A short computation shows that the expansion of
  $V_1(z) = y^{\frac{1}{2}} \theta(Dz) \overline{\theta(z)}$ at the cusp
  $\mathfrak{a}$ is always a constant multiple of $y^{\frac{1}{2}} \theta(uz)
    \overline{\theta(vz)}$ for some integers $u,v$ with $uv = D$.
  By Proposition~\ref{prop:Fricke-regularizable}, the
  non-exponential decay portion of $V_1(z)$ perfectly cancels with the
  Eisenstein series $E(z, s)$ at every singular cusp $\mathfrak{a}$.
  As such, it suffices to examine the exponential decay part of the constant
  term of $y^{\frac{1}{2}}
    \theta(uz) \overline{\theta(vz)}$, which is
  \[
    \sum_{\substack{m,n \geq 1 \\ um = vn}} r_1(m) r_1(n) e^{-2\pi (um + vn)y}.
  \]
  Since $r_1$ restricts to squares, the condition $um=vn$ forces $u/v$ to be a
  rational square, which contradicts that $uv=D$ is square-free.
  Hence the constant term of $V(\sigma_\mathfrak{a} z)$ is identically $0$,
  implying that the inner product is zero as well.
\end{proof}

\subsection{Vanishing of Most of the Discrete Spectrum}\label{ssec:disc}

We turn now to the discrete spectrum.
Recall that an automorphic form $\mu \in L^2(\Gamma_0(4D) \backslash \calH;
  \chid)$ is \emph{dihedral} if there exists a Hecke character $\eta$ of
$\Q(\sqrt{D})$ such that $L(s, \mu) = L(s, \eta)$.

We claim that $\langle V, \mu \rangle = 0$ unless $\mu$ is dihedral.
If $\mu$ were a newform, justification for this claim can be obtained by
mimicking the proof of Lemma~4.4 in~\cite{hkldw_3aps}.

\begin{proof}[(Proof sketch when $\mu$ is a newform)]
  The theta function $\theta(z)$ appears in the residue of a weight $1/2$, level
  $4D$ Eisenstein series
  \[
    E(z,w; \Gamma_0(4D))
    \colonequals
    \sum_{\gamma \in \Gamma_{\infty} \backslash \Gamma_0(4D)} \Im(\gamma z)^w J(\gamma, z)^{-1},
  \]
  where $j(\gamma, z) = \theta(\gamma z) / \theta(z)$ and $J(\gamma, z) =
    j(\gamma, z) / | j(\gamma, z) |$ is a normalized theta multiplier.
  This Eisenstein series has meromorphic continuation to $w \in \mathbb{C}$ with a simple pole at $w=\frac{3}{4}$ with residue of the form $c^{-1} y^{1/4} \theta(z)$.
  By computing the constant term in the Fourier expansion of
  $E(z,w;\Gamma_0(4D))$ (similar to the computations
  from~\cite[\S3.1]{goldfeld2006automorphic}), we conclude that
  \[
    c = \frac{4 \pi D}{3} \prod_{p \mid 4D} (1 + p^{-1}).
  \]
  We compute $\langle V, \mu \rangle$ by identifying $\theta(z)$ as the
  residue of $E(z,w;\Gamma_0(4D))$ and using this Eisenstein series to unfold
  the integral:
  \begin{align}\label{eq:unfolding-theta-pre}
    \langle V, \mu \rangle
     & = \langle y^{\frac{1}{2}} \theta(Dz) \overline{\theta(z)}, \mu \rangle
    = c \Res_{w = \frac{3}{4}}
    \langle y^{1/4} \theta(Dz) \overline{E(z,\overline{w};\Gamma_0(4D))}, \mu \rangle    \\
     & = c \Res_{w = \frac{3}{4}} \int_{0}^{\infty} \int_{0}^{1}
    y^{1/4} \theta(Dz) y^{w} \overline{\mu(z)} \frac{dxdy}{y^2}                          \\
     & = c \Res_{w = \frac{3}{4}} \sum_{n \ge 1} r_1(n) \overline{\rho(Dn)}
    \int_0^{\infty} y^{w-\frac{1}{4}} K_{i t} (2D \pi n y) e^{-2 D \pi n y} \frac{dy}{y} \\
     & = c \cdot  2\sqrt{\pi} \Res_{w = \frac{3}{4}}
    \frac{\Gamma(w-\frac{1}{4}+it) \Gamma(w-\frac{1}{4} - i t)}
    {(4 \pi D)^{w-\frac{1}{4}} \Gamma(w + \frac{1}{4})}
    \sum_{n \ge 1} \frac{\overline{\rho(Dn^2)}}{n^{2w-\frac{1}{2}}},
  \end{align}
  in which we've used~\cite[6.621(3)]{GradshteynRyzhik07} to evaluate the integral.
  As the gamma functions are analytic at $w = \frac{3}{4}$, the residue is zero
  unless the Dirichlet series has a pole at $w=\frac{3}{4}$.

  If $\mu$ is a Hecke newform, this Dirichlet series is essentially the symmetric
  square $L$-function associated to $\mu$; it has a pole if and only if $\mu$
  is dihedral.
\end{proof}

If all relevant Maass forms were newforms, this proof idea would be complete.
But when $D \equiv 1 \bmod 4$, it is necessary to consider oldforms too.

\subsection*{Oldforms and Newforms}

For a positive integer $N$ and a Dirichlet character
$\chi: (\Z / N \Z)^{\times} \to \C$, we consider $\Szero{M}{\chi}$ and $\Sone{M}$ for
divisors $M$ of $N$.
Recall that we have the decomposition (e.g.~\cite[Thm~5.8.3]{diamond2005first}
for holomorphic modular forms)
\[
  \Sone{N} = \bigoplus_{M \mid N} i_M\bigl(\Sone{M}^{\new}\bigr),
\]
where $i_M: \Sone{M} \to \Sone{N}$ is given
by $i_M = \bigoplus_{a \mid (N/M)} B_a$ with
\begin{equation}\label{eq: level raising definition}
  (B_a f)(z) \colonequals \tfrac{1}{\sqrt{a}} f(az).
\end{equation}
Decomposing according to nebentypus, we obtain
\[
  \Szero{N}{\chi}
  =
  \bigoplus_{\cond(\chi) \mid M \mid N} i_M\bigl(\Szero{M}{\chi}^{\new}\bigr).
\]

When $N = 4D$ and $\chi = \chi_{4D}$, the conductor of $\chi_{4D}$ is
$\cond(\chi_{4D}) = q = \tfrac{4D}{\ell}$, hence
\[
  \Szero{4D}{\chi_{4D}}
  =
  \bigoplus_{a \mid \ell} i_{aq}\bigl(\Szero{aq}{\chi_{aq}}^{\new}\bigr).
\]
Here, this means that
\begin{equation}\label{eq:S0new}
  \Szero{4D}{\chid}
  = \begin{cases}
    \Szero{4D}{\chid}^{\new} & D \equiv 2,3 \bmod 4 \\
    \\
    \begin{aligned}
       & \Szero{4D}{\chi_{4D}}^{\new}
      \\
       & \oplus i_2\bigl(\Szero{2D}{\chi_{2D}}^{\new}\bigr)
      \\
       & \oplus i_4\bigl(\Szero{D}{\chi_D}^{\new}\bigr)
    \end{aligned}
                             & D \equiv 1 \bmod 4.
  \end{cases}
\end{equation}

To understand the discrete spectrum, it is therefore necessary to consider terms coming from both newforms and oldforms.

\begin{lemma}\label{lem:old_form_inner}
  Let $a,b$ be positive integers such that $ab \mid \ell$.
  Let $\mu$ denote a Hecke-Maass form in $\Szero{bq}{\chi_{bq}}^{\new}$ which is an
  eigenform of spectral type $\tfrac{1}{2}+it$ for the Hecke algebra with associated Hecke eigenvalues $\lambda(n)$
  and Fourier coefficients $\rho(n)$. Then $\langle V, B_a \mu \rangle = 0$ unless $\mu$ is dihedral, in which case
  \begin{equation}
    \left(1 - \frac{\chi_q(2)}{2} \right) \rho(1) \frac{\langle V, B_a \mu
      \rangle_{4D}}{\langle \mu, \mu \rangle_{4D}}
    = \left(1 - \frac{\chi_{bq}(2)}{2} \right) \cdot \frac{P_2(1,B_a \mu)}{P_2(1, \mu)} \cdot
    \frac{ 2 \lambda(D)}{ \sqrt{\ell} h(D) \log \varepsilon},
  \end{equation}
  where
  \[
    P_2(s, \mu) = \sum_{j=0}^{\infty} \frac{\rho(2^{2j})}{2^{js}}
  \]
  is the $2$-factor of the symmetric square $L$-function associated to $\mu$.
\end{lemma}

Note that when $b = \ell$ and $a = 1$, we essentially have the heuristic proof
and the proof from~\cite{hkldw_3aps} described above.

\begin{proof}
  Let
  \[
    E(z,w; \Gamma_0(4D)) \colonequals \sum_{\gamma \in \Gamma_{\infty}
      \backslash \Gamma_0(4D)} \im(\gamma z)^w J(\gamma, z)^{-1},
  \]
  where $j(\gamma, z) = \theta(\gamma z) / \theta(z)$ and
  $J(\gamma, z) = j(\gamma, z)/ \lvert j(\gamma, z)\rvert$.
  As noted above, we recognize $\theta$ as coming from a residue of a
  half-integral weight Eisenstein series,
  \[
    \Res_{w = \frac{3}{4}} E(z,w; \Gamma_0(4D)) = c^{-1} y^{1/4} \theta(z),
  \]
  with
  \[
    c = \frac{4 \pi D}{3} \prod_{p \mid 4D} (1 + p^{-1}).
  \]
  Unfolding and recalling that the Fourier coefficients of $B_a \mu$ satisfy
  $\rho_{B_a \mu}(n) = \rho_{\mu}(n / a)$ shows that
  \[
    \langle V, B_a \mu \rangle_{4D} = c \cdot 2 \sqrt{\pi} \Res_{w = \frac{3}{4}}
    \frac{\Gamma(w - \tfrac{1}{4} + it) \Gamma(w - \tfrac{1}{4} - it)}
    {(4 \pi D)^{w - \tfrac{1}{4}}\Gamma(w + \tfrac{1}{4})}
    \sum_{n \ge 1}
    \frac{\overline{\rho(Dn^2/a)}}{n^{2w-\tfrac{1}{2}}}.
  \]

  Let $\{T_n \}$ be the Hecke operators on $S_0(D, \chi_{D})$, as defined
  in~\cite[\S6]{dfi}.
  Then for all $D$ and $n$ we have
  \[
    T_{D} T_{n} = \sum_{d \mid (D, n)} \chi_{D}(d) T_{Dn/d^2} = T_{Dn},
  \]
  since $\chi_{D}(d) = 0$ for all $1 \ne d \mid D$.
  It follows that $\lambda(Dn) = \lambda(D) \lambda(n)$, and as a consequence
  $\rho(Dn) = \rho(1) \lambda(D) \lambda(n)$. When $2 \mid D$, we must have
  $b = 1$ and $a = 1$;
  hence in all cases we have $\rho(Dn^2/a) = \rho(D) \lambda(n^2/a)$.
  Therefore we get
  \[
    \sum_{n=1}^{\infty} \frac{\overline{\rho(Dn^2/a)}}{n^s}
    = \overline{\lambda(D)} \sum_{n=1}^{\infty}
    \frac{\overline{\rho(n^2/a)}}{n^s}
    = \overline{\lambda(D)} \sum_{r=0}^{\infty} \sum_{2 \nmid n}
    \frac{\overline{\rho(n^2 \cdot  2^{2r}/a)}}{n^s \cdot 2^{rs}}.
  \]
  Since $a \mid 4$, we can use multiplicativity to write
  \[
    \sum_{n=1}^{\infty} \frac{\overline{\rho(Dn^2/a)}}{n^s}
    =  \overline{\lambda(D)} \sum_{j=0}^{\infty} \frac{\overline{\rho(2^{2j}/a)}}{2^{js}}
    \sum_{2 \nmid n} \frac{\overline{\lambda(n^2)}}{n^s}
    =  \overline{\rho(D)}
    \frac{P_2(s, B_a \overline{\mu})}{P_2(s, \overline{\mu})}
    \sum_{n=1}^{\infty} \frac{\overline{\lambda(n^2)}}{n^s},
  \]
  where $P_2(s, \mu) = \sum_{j=0}^{\infty} \frac{\rho(2^{2j})}{2^{js}}$ as
  in the statement of the lemma.

  Recalling that
  \[
    L(s, \overline{\mu} \otimes \overline{\mu})
    = \zeta^{(bq)}(2s) \sum_{n=1}^{\infty} \frac{\overline{\lambda(n)}^2}{n^s}
    = \zeta^{(bq)}(2s) L(s, \chi_{bq}) \sum_{n=1}^{\infty}  \frac{\overline{\lambda(n^2)}}{n^s},
  \]
  we may write
  \[
    \sum_{n=1}^{\infty} \frac{\overline{\rho(Dn^2/a)}}{n^s}
    = \overline{\rho(D)}
    \frac{P_2(s, B_a \overline{\mu})}{P_2(s, \overline{\mu})}
    \frac{L(s, \overline{\mu} \otimes \overline{\mu})}{\zeta^{bq}(2s) L(s,
      \chi_{bq})}.
  \]
  This has a pole at $s = 1$ if and only if $\mu$ is dihedral.\footnote{The
    authors appreciate insight on this from
    MathOverflow~\cite{moquestion_dihedral}, which led to us reconsidering
    foundational work of Labesse and Langlands~\cite{ll1979}.}
  It also follows that
  \begin{equation}\label{eq:inner_product_with_Ba}
    \langle V, B_a \mu \rangle_{4D}
    = \frac{P_2(1, B_a \overline{\mu})}{P_2(1, \overline{\mu})} \langle V, \mu
    \rangle_{4D}.
  \end{equation}

  We now assume that $\mu$ is dihedral and compute the inner product.
  Then
  \[
    \Res_{s=1} \sum_{n=1}^{\infty} \frac{|\rho(n)|^2}{n^s}
    = \frac{4 \cosh(\pi t)}{\pi}
    \langle |\mu|^2, \Res_{s=1} \mathcal{E}(z,s)
    \rangle_{4D}
    = \frac{3 \cosh(\pi t) \langle \mu, \mu \rangle_{4D}}{\pi^2 D \prod_{p
        \mid 4D} (1 + \tfrac{1}{p})},
  \]
  where
  \[
    \mathcal{E}(z,s)
    = \sum_{\gamma \in \Gamma_{\infty} \backslash \Gamma_0(4D)} \im(\gamma z)^s.
  \]
  By Dirichlet's class number formula, we have
  \begin{align}\label{eq:class-number-formula}
    L(1, \chi_{q}) = \frac{2 h(D) \log \varepsilon}{\sqrt{q}}
    = \frac{h(D) \sqrt{\ell} \log \varepsilon }{\sqrt{D}}.
  \end{align}
  As $b \in \{1, 2, 4\}$, we have
  \[
    \left( 1 - \frac{\chi_{bq}(2)}{2} \right)
    L(1, \chi_{bq}) = \left( 1 - \frac{\chi_{q}(2)}{2} \right)
    \frac{ h(D) \log \varepsilon \sqrt{\ell}}{\sqrt{D}}.
  \]
  Consequently,
  \[
    \left( 1 - \frac{\chi_{q}(2)}{2} \right) \Res_{s=1} \rho(1) \sum_{n=1}^{\infty} \frac{\overline{\rho(Dn^2)}}{n^s}
    = 
    \frac{(1-\frac{\chi_{bq}(2)}{2}) \cdot 3 \overline{\lambda(D)} \cosh(\pi t) \langle \mu, \mu \rangle_{4D}}
    {\pi^2 \sqrt{D\ell} h(D) \log \varepsilon \prod_{p \mid 4D} (1 + \tfrac{1}{p})}.
  \]
  After substituting, we obtain
  \begin{align*}
    \left( 1 - \frac{\chi_{q}(2)}{2} \right) \rho(1) \langle V, \mu \rangle_{4D}
     & = \left( 1 - \frac{\chi_{bq}(2)}{2} \right) \frac{3c \overline{\lambda(D)} \langle \mu, \mu \rangle_{4D}}
    {2 \sqrt{\ell} \pi D h(D) \log \varepsilon \prod_{p \mid 4D} (1 + \tfrac{1}{p})}                              \\
     & = \left( 1 - \frac{\chi_{bq}(2)}{2} \right) \frac{  2 \overline{\lambda(D)} \langle \mu, \mu \rangle_{4D}}
    { \sqrt{\ell} h(D) \log \varepsilon}.
  \end{align*}
  Using~\eqref{eq:inner_product_with_Ba} we get
  \[
    \left( 1 - \frac{\chi_{q}(2)}{2} \right) \rho(1) \frac{\langle V, B_a \mu
      \rangle_{4D}}{\langle \mu, \mu \rangle_{4D}}
    =  \left( 1 - \frac{\chi_{bq}(2)}{2} \right) \cdot \frac{P_2(1, B_a \overline{\mu})}{P_2(1, \overline{\mu})} \cdot \frac{ 2 \overline{\lambda(D)}}
    {\sqrt{\ell} h(D) \log \varepsilon},
  \]
  which is the desired equality.
\end{proof}

\subsection{Orthogonalizing}\label{ssec:orth}

We've now shown that the only Maass forms that appear in the spectral
decomposition are linear combinations of lifts of dihedral newforms.
Fortunately, this situation isn't as complicated as it might appear.

\begin{lemma}\label{lem:onlysd}
  If $D \equiv 1 \bmod 4$, then neither of the spaces $\Szero{4D}{\chi_{4D}}^{\new}$ nor
  $\Szero{2D}{\chi_{2D}}^{\new}$ contain any dihedral Maass forms.
\end{lemma}

\begin{proof}
  Let $\eta$ be a Hecke character of $\Q(\sqrt{D})$, and assume by contradiction that $\mu_{\eta}$ has (minimal) level $2D$ or $4D$.
  Since $\mu_{\eta}$ has nebentypus $\chi_{D}$, it follows that $\eta \vert_{\Q^{\times}} = 1_{2}$. In particular, recalling also that $L_2(s, \mu_{\eta}) = L_2(s, \eta)$, the conductor of $\eta$ divides $\frakp_2$, a prime ideal over $2$.
  If $2$ is split or ramified in $\Q(\sqrt{D})$, then $\calO_D / \frakp_2 \simeq \F_2$, hence there is no character of conductor $\frakp_2$, showing that $\eta$ has trivial conductor, hence $\mu_{\eta}$ has minimal level $D$.

  When $2$ is inert in $\Q(\sqrt{D})$, we have $\calO_D / \frakp_2 \simeq \F_4$, and since $D \equiv 1 \bmod 4$, the image of $\varepsilon$ in $\F_4^{\times}$ is a generator, showing that the map $\calO_D^{\times} \to \{ \pm 1\}^2 \times \F_4^{\times}$ is surjective, and the narrow class group of level $2$ is isomorphic to the class group. Therefore, there are no Hecke characters of conductor $2$, hence in this case as well $\eta$ has trivial conductor, and $\mu_{\eta}$ has minimal level $D$.
\end{proof}

Thus we have two scenarios: if $D \not \equiv 1 \bmod 4$, then there are only
newforms to consider; if $D \equiv 1 \bmod 4$, then the only Maass forms that
don't vanish are lifts of dihedral forms on $S_0(D, \chi_D)$.
In particular, if $\{\mu_j\}$ is an orthonormal basis of eigenforms for
$L^2(\Gamma_0(D) \backslash \calH; \chi_D)$, then $\{ B_a \mu_j \}_{a \mid 4,j}$
is a basis (not necessarily orthonormal) for
$i_4(L^2(\Gamma_0(D) \backslash \calH; \chi_D))$.

In the spectral expansion~\eqref{eq:initial_spectral}, we have an orthonormal
basis of Maass forms that are either dihedral or lifts of dihedral forms.
We now describe these forms and their coefficients explicitly. We largely follow the notation from~\cite[\S1.7,\S1.9]{bump1998automorphic},
though we refer also to~\cite{lang1994algebraic}
and~\cite{molin2022computing}.

Each dihedral Maass form arises from a Hecke character.
Let $\eta$ be a Hecke character on $F = \Q(\sqrt{D})$, ramified only at
$\infty$. Then $\eta = \eta_0 \eta_{\infty}^m$ for some $m = m(\eta) \in \Z$,
with $\eta_0 \in \Cl^+(F)^{\vee}$ a finite order character of the narrow class
group and
\begin{equation}
  \eta_{\infty}(x) =
  \eta_{\infty}(x_1, x_2)
  =
  \mathrm{sgn}(x_1) \mathrm{sgn}(x_2)
  \Bigl \lvert \frac{x_1}{x_2} \Bigr \rvert^{\frac{i \pi}{2 \log \varepsilon}},
\end{equation}
where $x_1, x_2$ are the two embeddings of $x$ into $\mathbb{R}$ (after having
fixed the ordering of the embeddings).
For each such character $\eta$, the function
\begin{align}
  \mu_{\eta}(x+iy)
   & \colonequals
  \delta_{\eta=1} \cdot \frac{h(D) \log \varepsilon}{2}  \sqrt{y} \\
   & \qquad +
  \sum_{\frakb \subseteq \calO_D} \eta(\frakb)
  \sqrt{y} K_{\frac{m i \pi}{2 \log \varepsilon}} (2 \pi N(\frakb) y)
  \cdot
  \mathrm{cs}_m(2 \pi N(\frakb) x),
\end{align}
in which $\mathrm{cs}_m = \cos $ if $2 \mid m$ and $ \mathrm{cs}_m = \sin $ if
$ 2 \nmid m$, is a dihedral Maass form~\cite{Maass} for $\Gamma_0(4D)$ with
nebentypus $\chid$.
Note that $\mu_{\eta}(z) = \mu_{\eta^{-1}}(z)$. Moreover, $\mu_\eta$ is non-cuspidal
if and only if $\eta$ factors through the norm by~\cite[Cor.\ 6.6]{ll1979}.
If $\eta$ factors through the norm, then (using the same notation
$\chi_{r^\diamond}$ as in Remark~\ref{rem:explicit-regularization})
$\eta = \chi_{r^{\diamond}} \circ \Nm_{F/\Q}$ and direct calculation shows
$\mu_{\eta} = h(D) \log \varepsilon \sqrt{r} E_{\chi_{r^{\diamond}}, \chid
      \chi_{r^{\diamond}}}$.

We note that factoring through the norm is equivalent to $\eta = \eta^{-1}$, as
\[
  \eta^{-1}(\frakb) = \eta(\overline{\frakb})
  \iff
  1 = \eta(\frakb \overline{\frakb}),
\]
hence $\eta = \eta^{-1}$ if and only if $\eta(\frakb) =
  \eta(\overline{\frakb})$ for all nontrivial $\frakb$.
Equivalently, $\eta$ factors through the norm if and only if $\eta$ is fixed
by the natural action of $\Gal(F/ \Q)$.
Let $X_F$ be a set of representatives for $\Cl^+(F)^{\vee} / \Gal(F / \Q)$
excluding the fixed points, i.e.\ such that $\eta \ne \eta^{-1}$.
Then $\lvert X_F \rvert = \tfrac{h^+(D) - d(D)/2}{2}$.

The dihedral Maass forms $\mu_\eta$ are eigenforms for the Hecke algebra with
Hecke eigenvalues $\eta(\frakp) + \eta^{-1}(\frakp)$ at the split primes
$(p) = \frakp \overline{\frakp}$.
Further, for every $a \mid \ell$, we have that $B_a \mu_{\eta} = B_a \mu_{\eta^{-1}}$
is also a Hecke eigenform for the odd part
of the Hecke algebra of $\Gamma_0(\ell D)$.
The Hecke operators at $2$ require more work.
When $D \equiv 1 \bmod 4$ (so that $\ell = 4$), for any
$f \in L^2(\Gamma_0(D) \backslash \calH ; \chi_{D})$
we have
\[
  T_2 B_1 f = T_2 f - \chi_D(2) B_2 f, \qquad
  T_2 B_2 f = f, \qquad
  T_2 B_4 f = B_2 f.
\]

It follows that when $f$ is a Hecke eigenform with $T_2 f = \lambda(2) f$,
and $U_f$ is the subspace spanned by $f, B_2 f, B_4 f$, then $T_2 \vert_{U_f}$
is represented by the matrix
\[
  \begin{pmatrix}
    \lambda(2) & 1 & 0 \\
    -\chi_D(2) & 0 & 1 \\
    0          & 0 & 0
  \end{pmatrix}
\]
with respect to the basis $f, B_2 f, B_4 f$.

For any $p \ne 2$, $T_p$ commutes with $B_a$ for each $a \mid 4$, which implies
that $\langle B_{a_1} \mu_i, B_{a_2} \mu_j \rangle = 0$ whenever $\mu_i$ and
$\mu_j$ are distinct dihedral Maass forms on $\Gamma_0(D)$.
Thus the space of dihedral forms admits the orthogonal decomposition
\[
  i_4(L_{\textup{dihedral}}^2(\Gamma_0(D) \backslash \calH; \chi_D))
  =
  \bigoplus_{j} (\C \mu_j \oplus \C B_2 \mu_j \oplus \C B_4 \mu_j),
\]
where $\mu_j$ range over the dihedral forms coming from Hecke characters $\eta$
that do not factor through the norm.
Let $U_j = \C \mu_j \oplus \C B_2 \mu_j \oplus \C B_4 \mu_j$.
Although $T_2$ is no longer normal, it satisfies $T_2^* = W T_2 W^{-1}$, where
$W$ is the Fricke involution (cf.~\cite[\S6]{dfi}).
With respect to the basis $\{\mu_j, B_2 \mu_j, B_4 \mu_j \}$ we have
\[
  W \vert_{U_j} = \begin{pmatrix}
    0 & 0 & \frac{1}{2} \\
    0 & 1 & 0           \\
    2 & 0 & 0
  \end{pmatrix},
  \quad
  T_2 \vert_{U_j} = \begin{pmatrix}
    \lambda_j(2) & 1 & 0 \\
    -\chi_D(2)   & 0 & 1 \\
    0            & 0 & 0
  \end{pmatrix}.
\]

Therefore
\[
  T_2^* \vert_{U_j} = \begin{pmatrix}
    0 & 0 & 0                      \\
    2 & 0 & -\frac{1}{2} \chi_D(2) \\
    0 & 2 & \lambda_j(2)
  \end{pmatrix}.
\]

Let $G$ be the Gram matrix of $U_j$.
This satisfies $T_2^t G = G T_2^*$, giving rise to a linear system of equations
with solution
\[
  G = \frac{1}{4}\langle \mu_j, \mu_j \rangle \begin{pmatrix}
    4                                                  & \frac{4 \lambda_j(2)}{2 + \chi_D(2)} & \frac{2 \lambda_j^2(2)}{2 + \chi_D(2)} - \chi_D(2) \\
    \frac{4 \lambda_j(2)}{2 + \chi_D(2)}               & 2                                    & \frac{2 \lambda_j(2)}{2 + \chi_D(2)}               \\
    \frac{2 \lambda_j^2(2)}{2 + \chi_D(2)} - \chi_D(2) & \frac{2 \lambda_j(2)}{2 + \chi_D(2)} & 1
  \end{pmatrix}.
\]
From this, it is straightforward to compute that
\begin{equation}
  \begin{split}\label{eq:basis}
    \mu_{j,1} & = \mu_j, \qquad \mu_{j,2} = B_2 \mu_j - \frac{\lambda_j(2)}{2 +
    \chi_D(2)} \mu_j,                                                                            \\
    \mu_{j,3} & = B_4 \mu_j - \tfrac{1}{2} \lambda_j(2) B_2 \mu_j + \tfrac{1}{4} \chi_D(2) \mu_j
  \end{split}
\end{equation}
is an orthogonal basis with norms
\begin{equation}
  \begin{split}\label{eq:basis_norms}
    \frac{\langle  \mu_{j,1} , \mu_{j,1} \rangle}{\langle  \mu_j, \mu_j \rangle} & =
    1, \qquad \qquad
    \frac{\langle  \mu_{j,2}, \mu_{j,2} \rangle}{\langle  \mu_j, \mu_j \rangle} = \frac{1}{2} - \frac{\lambda_j(2)^2}{5+4\chi_D(2)},
    \\
    \frac{\langle  \mu_{j,3}, \mu_{j,3} \rangle}{\langle  \mu_j, \mu_j \rangle}  & = \frac{3}{16} - \frac{2 - \chi_D(2)}{8(2 + \chi_D(2))} \lambda_j(2)^2 .
  \end{split}
\end{equation}

\section{Explicitly computing remaining contributions}\label{sec:compute}

The spectral decomposition from~\eqref{eq:initial_spectral} now takes the form

\begin{equation}\label{eq:no_cont}
  \frac{4 \Gamma(s) \Zodd(2s)}{(4\pi D)^s}
  =
  \bigl\langle E(\cdot, \tfrac{1}{2}), P_\ell(\cdot, \overline{s}+\tfrac{1}{2}) \bigr\rangle
  +
  \sum_j \langle P_\ell(\cdot, \overline{s}+\tfrac{1}{2}), \mu_j \rangle
  \frac{\langle V, \mu_j \rangle}{\lVert \mu_j \rVert^2},
\end{equation}
where the continuous spectrum has vanished by \S\ref{ssec:cont} and the
discrete spectrum consists of dihedral Maass forms in $L^2(\Gamma_0(4D),
  \chid)$ when $D \equiv 2 \bmod 4$, or orthogonalized
lifts of dihedral Maass forms in $L^2(\Gamma_0(D), \chi_D)$ when $D \equiv 1
  \bmod 4$ (cf. Lemmas~\ref{lem:old_form_inner} and~\ref{lem:onlysd}, and
equation~\eqref{eq:S0new}).
We find it easier to work with an orthogonal basis and to later divide by the
norm $\lVert \mu_j \rVert^2$.

To complete our understanding of the meromorphic continuation, we compute these
inner products.

\subsection{Poincar\'e inner products}

The inner product $\langle E, P_\ell \rangle$ extracts the $\ell$th Fourier
coefficient of $E$, weighted by a ratio of gamma functions.
To see this, write the Eisenstein series $E(z, w)$ from Remark~\ref{rem:explicit-regularization} as
\begin{equation}
  E(z, w) = \rho(w; 0) y^w + \widetilde{\rho}(w; 0) y^{1-w}
  +
  \sum_{n \neq 0} \rho(w; n)
  \sqrt{y}
  K_{w - \frac{1}{2}}(2 \pi \lvert n \rvert y)
  e^{2 \pi i n x}.
\end{equation}
The Fourier coefficient $\rho(w;\ell)$ may be made explicit by writing $E(z,w)$ in terms of Eisenstein series attached to characters (cf.~\eqref{eq:eisenstein_lincomb}) and applying formulas from~\cite[Proposition 4.1]{Young}.
Performing these computations shows that $\rho(\tfrac{1}{2}; \ell)$ is given by
\begin{equation} \label{eq: constant term in inner product with Eisenstein series}
  \rho(\tfrac{1}{2};\ell)
  = \begin{cases}
    \frac{d(D/2)}{\sqrt{D} L(1,\chid)}, & \text{if $D \equiv 2 \bmod 4$,} \\
    \frac{d(D)}{\sqrt{D} L(1,\chid)},   & \text{if $D \equiv 1 \bmod 4$.}
  \end{cases}
\end{equation}

A standard unfolding argument shows that for $\Re s \gg 1$,
\begin{equation}\label{eq:E_Pell}
  \langle E(\cdot, \overline{s} + \tfrac{1}{2}), P_\ell \rangle
  =
  \rho(\tfrac{1}{2}; \ell)
  \int_0^\infty e^{-2 \pi \ell y}
  y^{s}
  K_{0}(2 \pi \ell y)
  \frac{dy}{y}
  =
  \frac{\sqrt{\pi} \rho(\tfrac{1}{2}; \ell)\Gamma(s)^2}{(4 \ell \pi)^s\Gamma(s
    + \tfrac{1}{2})},
\end{equation}
in which we've used~\cite[6.621(3)]{GradshteynRyzhik07} to evaluate the integral.
We further simplify by using the class number formula~\eqref{eq:class-number-formula}
to see
\begin{equation}\label{eq:inner_eis}
  \langle E(\cdot, \overline{s}+ \tfrac{1}{2}), P_\ell \rangle
  =
  \frac{\sqrt{\pi}}{(4 \ell \pi)^s \Gamma(s + \tfrac{1}{2}) \log \varepsilon}
  \cdot
  \frac{\Gamma(s)^2}{h(D)} \cdot \frac{d(D)}{2}.
\end{equation}
(This includes writing $d(D/2) = d(D) / \sqrt{4}$ when $D \equiv 2 \bmod 4$,
which we do to unify results).

A similar argument gives a formula for the inner product of the Poincar{\'e}
series with a Maass cusp form $\mu_j \in L^2(\Gamma_0(4D), \chid)$ of type
$\frac{1}{2} + it$ and Fourier coefficients $\rho(n)$ when $\Re s \gg 1$,
yielding
\begin{equation}\label{eq:poincare_inner}
  \langle P_\ell( \cdot, \overline{s} + \tfrac{1}{2}), \mu \rangle
  = \frac{\sqrt{\pi} \rho(\ell) \Gamma(s + it_j) \Gamma(s - it_j)}
  {(4 \ell \pi)^{s} \Gamma(s+\frac{1}{2})}.
\end{equation}

\subsection{Dihedral inner products with \texorpdfstring{$D \equiv 2 \bmod 4$}{D = 2 mod 4}}

When $D \equiv 2 \bmod 4$, we showed that $L^2(\Gamma_0(4D) \backslash
  \mathcal{H}; \chid) = L^2(\Gamma_0(4D) \backslash \mathcal{H} ;
  \chid)^{\textup{new}}$ in~\eqref{eq:S0new}.
Thus the Hecke operators diagonalize the space and the cuspidal dihedral Maass
forms coming from Hecke characters on $\mathbb{Q}(\sqrt{D})$ (described in
\S\ref{ssec:orth}) are orthogonal and are the only Maass forms with nonzero
inner products $\langle V, \mu_j \rangle$.

These inner products can be read from Lemma~\ref{lem:old_form_inner} with $a =
  b = \ell = 1$, which simplifies to
\begin{equation}\label{eq:v_mu_easy}
  \rho_j(1) \frac{\langle V, \mu_j \rangle}
  {\langle \mu_j, \mu_j \rangle}
  =
  \frac{2 \lambda_j(D)}{h(D) \log \varepsilon}.
\end{equation}

\subsection{Dihedral inner products with \texorpdfstring{$D \equiv 1 \bmod 4$}{D = 1 mod 4}}

When $D \equiv 1 \bmod 4$, for each dihedral Maass form $\mu_j$ in
$L^2(\Gamma_0(D) \backslash \mathcal{H}; \chi_D)$, there is a three-dimensional
space of forms $V_j$ in $L^2(\Gamma_0(4D) \backslash \mathcal{H}; \chid)$ with
orthogonal basis $\mu_{j, 1}, \mu_{j, 2}$, and $\mu_{j, 3}$ given
in~\eqref{eq:basis}.

We recover the inner products from Lemma~\ref{lem:old_form_inner}, which gives
the inner products for the nonorthogonal basis $\mu_j, B_2 \mu_j$, and $B_4
  \mu_j$ with $\ell = 4$, $b = 1$, and $q = D$.
This gives
\begin{align*}
  \rho_j(1)
  \frac{\langle V, \mu_j \rangle}{\langle \mu_j, \mu_j \rangle}     & = \frac{\lambda(D)}{h(D) \log \varepsilon}                                          \\
  \rho_j(1)
  \frac{\langle V, B_2 \mu_j \rangle}{\langle \mu_j, \mu_j \rangle} & = \frac{\lambda(D)}{h(D) \log \varepsilon} \cdot \frac{\lambda_j(2)}{2 + \chi_D(2)} \\
  \rho_j(1)
  \frac{\langle V, B_4 \mu_j \rangle}{\langle \mu_j, \mu_j \rangle} & = \frac{\lambda(D)}{h(D) \log \varepsilon} \cdot \frac{1}{2}.
\end{align*}
Taking the linear combinations from~\eqref{eq:basis}, we find
\begin{align*}
  \rho_j(1)
  \frac{\langle V, \mu_{j,1} \rangle}{\langle \mu_j, \mu_j \rangle}
   & =
  \frac{\lambda(D)}{h(D) \log \varepsilon},
  \qquad \qquad
  \rho_j(1)
  \frac{\langle V, \mu_{j,2} \rangle}{\langle \mu_j, \mu_j \rangle}
  =
  0
  \\
  \rho_j(1)
  \frac{\langle V, \mu_{j,3} \rangle}{\langle \mu_j, \mu_j \rangle}
   & =
  \frac{\lambda(D)}{h(D) \log \varepsilon}
  \cdot
  \left(\frac{2+\chi_D(2)}{4} - \frac{\lambda_j(2)^2}{2(2+\chi_D(2))} \right).
\end{align*}
Using the norms computed in~\eqref{eq:basis_norms} and noting that $(2 + \chi_D(2))(2
  - \chi_D(2)) = 3$, we compute
\begin{equation}
  \begin{split}\label{eq:orthog_inner_products}
    \rho_j(1)
    \frac{\langle V, \mu_{j,1} \rangle}{\langle \mu_{j,1}, \mu_{j,1}
    \rangle}                                                                  & = \frac{\lambda(D)}{h(D) \log \varepsilon}, \qquad
    \rho_j(1)
    \frac{\langle V, \mu_{j,2} \rangle}{\langle \mu_{j,2}, \mu_{j,2} \rangle} = 0                                                                         \\
    \rho_j(1)
    \frac{\langle V, \mu_{j,3} \rangle}{\langle \mu_{j,3}, \mu_{j,3} \rangle} & = \frac{\lambda(D)}{h(D) \log \varepsilon} \cdot \frac{4}{2 - \chi_D(2)},
  \end{split}
\end{equation}

\subsection{Fourier coefficients}

The inner product with the Poincar\'e series~\eqref{eq:poincare_inner}
requires the $\ell$th Fourier coefficient of each Maass form.
When $D \equiv 2 \bmod 4$, we have $\ell = 1$, so our formulas are naturally stated in terms of $\rho_j(1)$.

But when $D \equiv 1 \bmod 4$, we have $\ell = 4$ and we need the coefficients
$\rho_{j, 1}(4)$ and $\rho_{j, 3}(4)$.
(We don't need $\rho_{j, 2}(4)$ as it is multiplied by $\langle V, \mu_{j, 2}
  \rangle = 0$, which is $0$ by~\eqref{eq:orthog_inner_products}).
For these, we again work in terms of the basis $\{ \mu_j, B_2 \mu_j, B_4 \mu_j
  \}$.
Recalling that $\mu_j$ is a Hecke newform on $\Gamma_0(D)$, we use Hecke
operators to compute the $4$th coefficient in terms of the Hecke eigenvalue
$\lambda_j(2)$, giving
\begin{align}
  \begin{split}\label{eq:mu_coeffs}
    \rho_{j,1}(4) & = \rho_j(4) = \rho_j(1)(\lambda_j(2)^2 - \chi_D(2))
    \\
    \rho_{j,3}(4) & = \rho_j(1) - \tfrac{1}{2} \lambda_j(2) \rho_j(2) + \tfrac{1}{4} \chi_D(2) \rho_j(4) \\
                  & = \tfrac{1}{4} \rho_j(1) (3 - (2-\chi_D(2)) \lambda_j(2)^2).
  \end{split}
\end{align}

\subsection{Assembly}

Finally, we are ready to assemble all the computations together.
Though there are differences between the cases $D \equiv 1 \bmod 4$ and $D
  \equiv 2 \bmod 4$, we can immediately bring these cases together.

When $D \equiv 2 \bmod 4$, we combine~\eqref{eq:poincare_inner}
and~\eqref{eq:v_mu_easy} to see that
\begin{equation}\label{eq:full_discrete}
  \sum_j \langle P_\ell(\cdot, \overline{s}+\tfrac{1}{2}), \mu_j \rangle
  \frac{\langle V, \mu_j \rangle}{\lVert \mu_j \rVert^2}
  =
  \sum_j
  \frac{\sqrt{\pi} \Gamma(s + it_j) \Gamma(s - it_j)}
  {(4 \ell \pi)^{s} \Gamma(s+\frac{1}{2})}
  \frac{2 \lambda_j(D)}{h(D) \log \varepsilon}.
\end{equation}
When $D \equiv 1 \bmod 4$, we first use~\eqref{eq:orthog_inner_products}
and~\eqref{eq:mu_coeffs} to sum
\begin{align*}
  \sum_{k=1}^{3} \rho_{j,k}(4) & \frac{\langle V, \mu_{j,k} \rangle}{\langle \mu_{j,k}, \mu_{j,k} \rangle}
  \\
                               & = \frac{\lambda_j(D)}{h(D) \log \varepsilon} \cdot \left( \lambda_j(2)^2 - \chi_D(2) + \tfrac{3 - (2-\chi_D(2)) \lambda_j(2)^2}{2-\chi_D(2)} \right) \\
                               & = \frac{2 \lambda_j(D)}{h(D) \log \varepsilon}.
\end{align*}
The same computation above then shows that the total discrete contribution
exactly matches~\eqref{eq:full_discrete}.

Before summing over all $j$, we recall that the type $t_j$ of the dihedral form
$\mu_j$ is given by $t_m = \tfrac{\pi i m}{2 \log \varepsilon}$,
where $\eta = \eta_0 \eta_\infty^m$.
We sum over $m$ instead of $j$, but we must account for multiplicity.
For $m \neq 0$, there are $h^+(D)$ ($= h(D)$ as we take $N(\varepsilon) = -1$)
many characters $\eta_0$.
But when $m = 0$, only $\lvert X_F \rvert = \frac{h^+(D) - d(D)/2}{2}$ choices
are cuspidal (cf.\ \S\ref{ssec:orth}).

Note that as $D$ is squarefree, every prime $p$ dividing $D$ is ramified in
$\calO_D$ and $\sqrt{D}\calO_D$ is the unique ideal of norm $D$.
Moreover, this ideal is narrowly principal as it is generated by the totally
positive $\varepsilon \sqrt{D}$, and hence for a Hecke character $\eta$ we
compute
\begin{equation}\label{eq:eigenvalue_sqrtD}
  \eta(\sqrt{D} \calO_D)
  =
  \eta_{\infty}^m(\varepsilon \sqrt{D})
  =
  \left \lvert
  - \frac{\varepsilon \sqrt{D}}{\overline{\varepsilon} \sqrt{D}}
  \right \rvert^{\frac{i m \pi}{2 \log \varepsilon}}
  = \varepsilon^{\frac{i m \pi}{\log \varepsilon}} = e^{i m \pi} = (-1)^m.
\end{equation}
This implies that that $\lambda(D) = (-1)^m$.

We apply this to sum over $j$ by summing over $m$, obtaining
\begin{align*}
  \sum_j & \langle P_\ell(\cdot, \overline{s}+\tfrac{1}{2}), \mu_j \rangle
  \frac{\langle V, \mu_j \rangle}{\lVert \mu_j \rVert^2}
  \\
         & = \frac{2 \sqrt{\pi}}{(4 \pi \ell)^s \Gamma(s + \frac{1}{2}) \log \varepsilon}
  \left( \sum_{m = 1}^{\infty}(-1)^m \Gamma(s + it_m)\Gamma(s-it_m)
  + \frac{\lvert X_F \rvert}{h(D)} \Gamma(s)^2 \right)
  \\
         & =
  \frac{2 \sqrt{\pi}}{(4 \pi \ell)^s \Gamma(s + \frac{1}{2}) \log \varepsilon}
  \left( \sum_{m = 1}^{\infty}(-1)^m \Gamma(s + it_m)\Gamma(s-it_m)
  + \frac{1}{2}\left(1 - \frac{d(D)}{2h(D)} \right) \Gamma(s)^2 \right).
\end{align*}

By the reduced spectral expansion~\eqref{eq:no_cont}, we need only add $\langle
  E(\cdot, \tfrac{1}{2}), P_\ell \rangle$ to recover $4 \Gamma(s) \Zodd(2s)/ (4\pi D)^s$.
With the evaluation~\eqref{eq:inner_eis}, this gives
\begin{equation}\label{eq:final}
  \frac{4 \Gamma(s) \Zodd(2s)}{(4\pi D)^s}
  =
  \frac{\sqrt{\pi}}{(4 \pi \ell)^s \Gamma(s + \frac{1}{2}) \log \varepsilon}
  \left( 2\sum_{m = 1}^{\infty}(-1)^m \Gamma(s + it_m)\Gamma(s-it_m)
  + \Gamma(s)^2 \right).
\end{equation}
The gamma duplication formula
shows that
\begin{equation}
  \frac{(4D\pi)^s}{4 \Gamma(s)}
  \frac{\sqrt{\pi}}{(4 \pi \ell)^s \Gamma(s + \frac{1}{2})}
  = \frac{D^s \sqrt{\pi}}{4 \ell^s \Gamma(s) \Gamma(s+\tfrac{1}{2})}
  = \frac{\left(\frac{4D}{\ell}\right)^s}{8 \Gamma(2s)} = \frac{q^s}{8 \Gamma(2s)}.
\end{equation}
Multiplying~\eqref{eq:final} by $(4 \pi D)^s/4\Gamma(s)$, rearranging, and
taking $s$ in place of $2s$ thus gives the following theorem.

\begin{theorem}
  For $s \in \mathbb{C}$ away from the poles of the summands,
  \[
    \Zodd(s) = \frac{q^{s/2}}{8 \Gamma(s) \log \varepsilon}
    \sum_{m \in \mathbb{Z}} (-1)^m
    \Gamma\Big(\frac{s}{2} + \frac{\pi i m}{2\log \varepsilon}\Big)
    \Gamma\Big(\frac{s}{2} - \frac{\pi i m}{2\log \varepsilon}\Big),
  \]
  Thus $\Zodd(s)$ has meromorphic continuation to $s \in \mathbb{C}$, with simple poles at $s = - 2k + \frac{\pi i m}{\log \varepsilon}$
  for $m \in \mathbb{Z}$ and integral $k \geq 0$.
\end{theorem}

\subsection{Extended Remark on $\Zeven$}\label{rem:Z-even}
  The even-indexed Fibonacci zeta function may be written
  \[
    \Zeven(s) = \tfrac{1}{4} \sum_{n \geq 1} \frac{r_1(n) r_1(Dn+\ell)}{n^{s/2}}.
  \]
  By analogy, one might hope to understand $\Zeven(s)$ through the inner product
  $\langle y^{\frac{1}{2}} \theta(Dz) \overline{\theta(z)}, P_{-\ell}(z,
    \overline{s}; \chid) \rangle$, where $P_{-\ell}$ is the formal generalization
  of~\eqref{eq:P-h-definition} to $h=-\ell$.

  Unfortunately, the series $P_{-\ell}(z,s;\chid)$ diverges for all $z \in
    \mathcal{H}$. Nevertheless, a method for defining $P_{-\ell}$ through controlled
  limiting processes was presented in~\cite{hoffsteinhulse13}. With this, one may
  adapt the argument in the odd-indexed case to prove a meromorphic continuation
  for $\Zeven(s)$.

  In the left half-plane $\Re s < 0$, the resulting meromorphic continuation takes the form
  \[
    \Zeven(s) = \frac{q^{\frac{s}{2}} \Gamma(1-s)}{4 \log \varepsilon}
    \sum_{m \in \mathbb{Z}} \frac{\Gamma(\frac{s}{2} - \frac{\pi i m}{2 \log \varepsilon})}{\Gamma(1-\frac{s}{2} - \frac{\pi i m}{2 \log \varepsilon})},
  \]
  which exactly matches a continuation established for $\Zeven$ in~\cite{akldwFibonacciGeneral} via Poisson summation. To see this, one combines equations~\cite[(3.29)]{hoffsteinhulse13} and~\cite[(4.17)]{hoffsteinhulse13}, simplifying further using the gamma function identities
  \begin{align*}
    \frac{\Gamma(\frac{1}{2}- \frac{s}{2})}{\Gamma(\frac{s}{2})}
     & = \frac{2^s \Gamma(1-s)}{\sqrt{\pi} \csc(\frac{\pi s}{2})},       \\
    \frac{2 \Gamma(\frac{s}{2} - it_j)\Gamma(\frac{s}{2} + it_j)}{\Gamma(\frac{1}{2} - it_j) \Gamma(\frac{1}{2} + it_j) \csc(\frac{\pi s}{2})}
     & = \frac{\Gamma(\frac{s}{2} - it_j)}{\Gamma(1-\frac{s}{2} - it_j)}
    + \frac{\Gamma(\frac{s}{2} + it_j)}{\Gamma(1-\frac{s}{2} + it_j)}.
  \end{align*}
  (As in~\cite[\S{4.2}]{akldwFibonacciGeneral}, behavior on the line $\Re s = 0$ is more complicated.)

\bibliographystyle{alpha}
\bibliography{bibfile}

\end{document}